\documentclass{amsart}
\usepackage{mathrsfs}
\usepackage{enumerate}

\usepackage{easymat}

\numberwithin{equation}{section}

\newtheorem{theorem}{Theorem}[section]
\newtheorem{corollary}[theorem]{Corollary}
\newtheorem{lemma}[theorem]{Lemma}
\newtheorem{proposition}[theorem]{Proposition}
\theoremstyle{definition}
\newtheorem{definition}[theorem]{Definition}

\newtheorem{assumption}[theorem]{Assumption}

\newtheorem{remark}[theorem]{Remark}

\def\NN{{\mathbb{N}}}
\def\ee{\mathrm{e}}

\def\RR{\mathbb{R}}
\def\CC{\mathbb{C}}

\def\Ell{\mathrm{L}}
\def\C{\mathrm{C}}
\def\W{\mathrm{W}}

\def\vect#1#2{{\textstyle\binom{#1}{#2}}}
\def\bigvect#1#2{{\displaystyle\binom{#1}{#2}}}

\def\dom{D}
\def\ran{\mathop{\mathrm{rg}\mathstrut}}
\def\Re{\mathop{\mathrm{Re}}}
\def\dd{\:\mathrm{d}}
\def\calA{\mathcal{A}}
\def\calA{\mathcal{A}}

\def\calC{\mathcal{C}}
\def\calD{\mathcal{D}}
\def\calE{\mathcal{E}}

\def\calL{\mathcal{L}}

\def\calO{\mathcal{O}}
\def\calP{\mathcal{P}}

\def\calS{\mathcal{S}}
\def\calT{\mathcal{T}}

\def\calE{\mathcal{E}}

\newcommand{\Id}{Id}
\newcommand{\Am}{A_m}

\newcommand{\dE}{\partial E}
\newcommand{\Fav}{\mathrm{Fav}}
\newcommand{\FeA}{{\Fav_1(A)}}
\newcommand{\FnA}{{\Fav_0(A)}}

\newcommand{\Ttt}{(T(t))_{t\ge0}}
\newcommand{\Stt}{(S(t))_{t\ge0}}
\newcommand{\Tme}{T_{-1}}

\newcommand{\Tmett}{(\Tme(t))_{t\ge0}}
\newcommand{\cTt}{\calT(t)}
\newcommand{\cTCt}{\calT_C(t)}
\newcommand{\cTtt}{(\cTt)_{t\geq 0}}
\newcommand{\cTCtt}{(\cTCt)_{t\geq 0}}

\newcommand{\cSt}{\calS(t)}
\newcommand{\cStt}{(\cSt)_{t\geq 0}}
\newcommand{\cTitt}{(\calT_i(t))_{t\geq 0}}
\newcommand{\N}{|\!|\!|}
\newcommand{\blN}{\bigl|\!\bigl|\!\bigl|}
\newcommand{\brN}{\bigr|\!\bigr|\!\bigr|}
\newcommand{\calx}{{\scriptstyle{\mathcal X}}}
\newcommand{\calu}{{\scriptstyle{\mathcal U}}}

\newenvironment{smatrix}{\left(\begin{smallmatrix}}{\end{smallmatrix}\right)}

\begin{document}

\title[Product Formulas for Operator Matrices]{Stability and Convergence of\\ Product Formulas for Operator Matrices}

\author[A. B\'{a}tkai]{Andr\'{a}s B\'{a}tkai}
\address{A. B., E\"{o}tv\"{o}s Lor\'{a}nd University, Institute of Mathematics and Numerical Analysis and Large Networks Research Group, Hungarian Academy of Sciences, 1117 Budapest, P\'{a}zm\'{a}ny P. s\'{e}t\'{a}ny 1/C, Hungary.}
\email{batka@cs.elte.hu}

\author[P. Csom\'{o}s]{Petra Csom\'{o}s}
\address{Universit\"at Innsbruck, Institut f\"ur Mathematik, Technikerstra{\ss}e 13, 6020 Innsbruck, Austria}
\email{petra.csomos@uibk.ac.at}

\author[K.-J. Engel]{Klaus-Jochen Engel}
\address{University of L'Aquila, Department of Pure and Applied Mathematics, via Vetoio 1, 67010 L'Aquila, Italy}
\email{engel@ing.univaq.it}

\author[B. Farkas]{B\'{a}lint Farkas}
\address{E\"{o}tv\"{o}s Lor\'{a}nd University, Institute of Mathematics and Numerical Analysis and Large Networks Research Group, Hungarian Academy of Sciences, 1117 Budapest, P\'{a}zm\'{a}ny P. s\'{e}t\'{a}ny 1/C, Hungary.}
\email{fbalint@cs.elte.hu}



\date\today
\begin{abstract}
We present easy to verify conditions implying stability estimates for operator matrix splittings which ensure convergence of the associated Trotter, Strang and weighted product formulas. The results are applied to inhomogeneous abstract Cauchy problems and to boundary feedback systems.
\end{abstract}
\maketitle

\section{Introduction}
\noindent
Many systems in physics, biology or engineering can be described by an abstract Cauchy problem of the form
\[\tag{ACP}\label{ACP}
\begin{cases}
\tfrac{\dd}{\dd t}\calu(t)=\calA\,\calu(t)\quad\text{for } t\ge0,\\
\calu(0)=\calu_0
\end{cases}
\]
on a product $\calE=E\times F$ of two Banach spaces $E$ and $F$, see B\'{a}tkai and Piazzera \cite{Batkai-Piazzera}, Engel and Nagel \cite[Chapter VI]{EN:00}, Casarino et al.~\cite{CENN:03}, or Tretter \cite{Tr:08}.  By Engel, Nagel \cite[Section II.6]{EN:00} the problem \eqref{ACP} is well-posed if and only if the system operator $\calA$ is the generator of a strongly continuous semigroup $\cTtt$ on $\calE$. Moreover, in this case the unique (mild) solution $\calu(\cdot)$ of \eqref{ACP} is given by
\[
\calu(t)=\calT(t)\calu_0.
\]
However, in general it is not possible to calculate the entries of
\[
\calT(t)=\bigl(T_{ij}(t)\bigr)_{2\times2}
\]
in terms of $\calA$ in order to obtain an explicit representation of the solution $\calu(\cdot)$. But as we will see below this can be achieved  in case $\calA$ has some special structure, e.g., if $\calA$ is of triangular form. The idea at this point is to split $\calA$ into (a sum of) simpler pieces,  for which it is possible to calculate the associated semigroup and then to use some kind of product formula to reassemble $\calT(t)$ from these pieces. This approach is made more precise in the following result.

\begin{theorem}\label{prop-splitting}
For $i=1,2$ let $\calA_i$ be the generator of the strongly continuous semigroup $(\calT_i(t))_{t\ge0}$  on the Banach space $\calE$. Suppose that $\calA:=\overline{\calA_1+\calA_2}$ is the generator of the strongly continuous semigroup $(\calT(t))_{t\ge0}$. Then the following assertions are true.
\begin{enumerate}[(i)]
\item If there exists $M\geq 1$ and $\omega\in \RR$ such that
\begin{equation}\label{eq:cond_stab_Lie}
\bigl\|\left(\calT_2(\tfrac tn)\calT_1(\tfrac tn)\right)^n\bigr\| \leq M \ee^{\omega t} \quad\text{for all }t\ge0\text{ and } n\in\NN,
\end{equation}
then for all $\calx\in\calE$
\begin{align}
\calT(t)\calx &= \lim_{n\to \infty} \left(\calT_2(\tfrac{t}{n})\calT_1(\tfrac{t}{n})\right)^n \calx, \text{ and}\label{eq:conv_Lie}\\
\calT(t)\calx &= \lim_{n\to \infty} \left(\calT_1(\tfrac{t}{2n})\calT_2(\tfrac{t}{n})\calT_1(\tfrac{t}{2n})\right)^n \calx. \label{eq:conv_Strang}
\end{align}
\item If there exists $M\geq 1$ and $\omega\in\RR$ such that
\begin{equation}\label{eq:cond_stab_weighted}
\left\|\frac{1}{2^n}\bigl(\calT_1(\tfrac tn)\calT_2(\tfrac tn) + \calT_2(\tfrac tn)\calT_1(\tfrac tn)\bigr)^n\right\| \leq M \ee^{\omega t} \quad\text{for all }t\ge0\text{ and } n\in\NN,
\end{equation}
then for all $\calx\in\calE$
\begin{equation}\label{eq:conv_weighted}
\calT(t)\calx = \lim_{n\to \infty} \frac{1}{2^n}\left(\calT_1(\tfrac{t}{n})\calT_2(\tfrac{t}{n}) + \calT_2(\tfrac{t}{n})\calT_1(\tfrac{t}{n})\right)^n\calx.
\end{equation}
\end{enumerate}
\end{theorem}

\noindent For the proofs we refer to Engel and Nagel \cite[Corollary~III.5.8]{EN:00}, Csom\'{o}s and Nickel \cite[Section~2]{Csomos-Nickel}, and B\'{a}tkai, Csom\'{o}s and Nickel \cite[Section~4]{Batkai-Csomos-Nickel}.

Product formulas like \eqref{eq:conv_Lie}, \eqref{eq:conv_Strang} and \eqref{eq:conv_weighted} have been applied to approximate the solution of a variety of complicated differential equations and are referred to as ``operator splitting'' in numerical analysis, see for example the monographs by Farag\'{o} and Havasi \cite{Farago-Havasi_book}, Holden et al.~\cite{Holden-Karlsen-Lie-Risebro} or Hundsdorfer and Verwer \cite{Hundsdorfer-Verwer}. The procedure described in Equation \eqref{eq:conv_Lie} is called the Trotter product formula, or sequential splitting. Equation \eqref{eq:conv_Strang} is called the Strang splitting, and Equation \eqref{eq:conv_weighted} is called the (symmetrically) weighted splitting or additive operator splitting. These and many other different procedures have been introduced to increase the order of convergence. In the finite dimensional setting, sequential splitting is of first order, while the other two are of second order. There are many more higher order methods in the literature, see Hairer, Lubich and Wanner \cite[Section III.5.4]{Hairer-Lubich-Wanner}, but we concentrate here on these three main cases since they are the most frequently used ones in applications.

Various generalizations of this procedure are possible but will not be considered in this paper. For non-autonomous versions of these product formulas we refer to B\'{a}tkai et al.~\cite{Batkai-Csomos-Farkas-Nickel1}. For the combined effect of spatial approximation and operator splitting see B\'{a}tkai, Csom\'{o}s and Nickel \cite{Batkai-Csomos-Nickel}, and for the combination of rational approximations, operator splitting and spatial approximation see B\'{a}tkai et al.~\cite{Batkai-Csomos-Farkas-Nickel2}.

The crucial hypothesis to achieve convergence of these splitting procedures are stability conditions like \eqref{eq:cond_stab_Lie} or \eqref{eq:cond_stab_weighted}. In case the semigroups involved are not quasi-contractive, it is in general very difficult to verify these conditions by explicit computations.

The aim of this paper is to address this problem for a special class of triangular matrix operator semigroups, which occur quite frequently in applications.
To this end, in Section \ref{sec:spl_mat} we investigate the stability of the Trotter, Strang and weighted product formulas for triangular operator matrices. To do that we first characterize generators of triangular operator matrix semigroups. Then we analyze the conditions \eqref{eq:cond_stab_Lie} and \eqref{eq:cond_stab_weighted} in the triangular case and give an abstract sufficient condition ensuring them. Finally, we show how extrapolated Favard classes can be used to obtain the desired estimates.
In Section \ref{sec:applications} we consider two classes of applications: Inhomogeneous abstract Cauchy problems and abstract boundary feedback systems.

In what follows we use the term ``semigroup'' to indicate a strongly continuous one-parameter semigroup of bounded linear operators, our main reference on this topic is Engel, Nagel \cite{EN:00}.

\section{Splitting for Operator Matrices}\label{sec:spl_mat}
\noindent
In this section we first characterize generators of triangular matrix semigroups. Then we present conditions implying stability for products of triangular operator matrix semigroups. Finally, we show how our main assumption on the growth of the off-diagonal elements of the matrix semigroup can be verified by the use of Favard classes.

\subsection{Characterization of Triangular Matrix Semigroups}

As mentioned already in the introduction, in general it is not possible to give an explicit matrix representation of a semigroup $\cTtt$ on a product space in terms of the entries of the associated generator. However, things get much simpler if we restrict our attention to matrices of triangular form. In order to characterize this class of operators we associate to an operator
\[\calA:\dom(\calA)\subseteq\calE\to\calE\]
defined on the product space $\calE=E\times F$ the operator $A:\dom(A)\subseteq E\to E$ defined by
\[
Ax:=\pi_1\bigl(\calA \vect x0\bigr)\quad\text{for }
x\in \dom(A):=\left\{z\in E:\vect z0\in \dom(\calA)\right\},
\]
where $\pi_i$ denotes the projection on the $i^{\text{th}}$ coordinate.
Moreover, we denote by $s(\calA)$ the spectral bound of $\calA$. With these notations the following result holds.

\begin{proposition}\label{char-triangular}
Let $\calA$ generate the semigroup $\calT=\cTtt$ on $\calE=E\times F$. Then $\calT$  has upper triangular form, i.e.
\begin{equation}\label{triangular-sgr}
\calT(t)=\begin{pmatrix}
T(t)& R(t)\\
0& S(t)
\end{pmatrix}\quad\text{for all }t\ge0,
\end{equation}
if and only if the following two conditions are satisfied.
\begin{enumerate}[(i)]
\item For all $x\in E$ with $\vect x0\in \dom(\calA)$ we have $\pi_2\bigl(\calA\vect x0\bigr)=0$.
\item There exists $\lambda\in\rho(A)$ satisfying $\Re\lambda>s(\calA)$.
\end{enumerate}
Moreover, in this case $\Ttt$ is a semigroup with generator $A$.
\end{proposition}

\begin{proof}
Note first that, if $\calT(t)$ has upper triangular form \eqref{triangular-sgr}, then the entries $T(t)$ and $S(t)$ form semigroups. Denote their generators by $\widetilde A$ and $B$, respectively. By taking the Laplace transform of $t\mapsto \calT(t)$ we obtain for $\lambda$ large that
$$
R(\lambda,\calA)=\begin{pmatrix}
R(\lambda, \widetilde A)&\star\\
0&R(\lambda,B)
\end{pmatrix}
,
$$
i.e., $R(\lambda,\calA)$ has upper triangular form for $\lambda$ large.
Conversely, if $R(\lambda,\calA)$ has triangular form for sufficiently large $\lambda$, the Post--Widder inversion formula (see Engel and Nagel \cite[Corollary III.5.5]{EN:00}) implies that $\calT(t)$ is upper triangular. Hence $\calT(t)$ is of upper triangular form for all $t\ge0$ if and only if $R(\lambda,\calA)$ is of upper triangular form for all $\lambda$ sufficiently large.
This is further equivalent to the fact that for \textit{some} $\lambda\in\CC$ satisfying $\Re\lambda>s(\calA)$ the resolvent has upper triangular form. To see this we note that for $|\lambda-\mu|<\|R(\lambda,\calA)\|^{-1}$ we have $\mu\in\rho(\calA)$ and
$$
R(\mu,\calA)=\sum_{n=0}^{+\infty}(\lambda-\mu)^nR(\lambda,\calA)^{n+1}.
$$
Here the right-hand side yields matrices of upper triangular form and by holomorphy of the resolvent map we conclude that $R(\mu,\calA)$ is of upper triangular form in the whole connected component of $\rho(\calA)$ which is unbounded to the right. After these
preparations we turn to the proof.

\medskip\noindent
Suppose that $\calT(t)$ has upper triangular form for all $t\geq0$ and take some $\vect x0\in\dom(\calA)$. Then we obtain $\pi_2\bigl(\calT(t)\vect x0-\vect x0\bigr)=0$ and (i) follows by the definition of the generator of a semigroup. To show (ii) we fix $\lambda\in \rho(\calA)$ sufficiently large such that
\[
R(\lambda,\calA)=\begin{pmatrix}
R_1&R_2\\
0&R_4
\end{pmatrix}.
\]
We prove that $R_1$ is the inverse of $\lambda-A$, i.e., $\lambda\in \rho(A)$ and $R_1=R(\lambda,A)$  (which also implies that $A$ is the generator of $(T(t))_{t\ge0}$). Indeed, for an arbitrary $x\in E$, we have
$$
R(\lambda, \calA)\vect x0=\vect{R_1x}{0}\in\dom(\calA),
$$
and hence by definition $R_1x\in \dom(A)$, i.e., $\ran R_1\subseteq \dom(A)$. Moreover,
$$
(\lambda-A)R_1x=\pi_1\Bigl((\lambda-\calA)\vect {R_1x}{0}\Bigr)=\pi_1\Bigl((\lambda-\calA) R(\lambda, \calA)\vect x0\Bigr)=x,
$$
i.e., $R_1$ is the right-inverse of $\lambda-A$. We show that it is also a left-inverse. For $x\in \dom(A)$ we have
\begin{align*}
R_1(\lambda-A)x
&=R_1\pi_1\Bigl((\lambda-\calA)\vect x0\Bigr)=\pi_1\biggl(R(\lambda,\calA) \bigvect{\pi_1\bigl((\lambda-\calA)\vect x0\bigr)}{0}\biggr),
\intertext{which, by validity of (i), further equals to}
&=\pi_1\Bigl(R(\lambda,\calA) (\lambda-\calA)\vect x0\Bigr)=\pi_1\vect x0=x.
\end{align*}
Summing up, $\lambda\in \rho(A)$, hence (ii) is true. Moreover, this implies that $A=\widetilde A$.

\medskip\noindent
Suppose now that (i) and (ii) are satisfied, and fix a $\lambda\in \rho(A)\cap \rho(\calA)$. We have to prove that
$$
R(\lambda,\calA)=\begin{pmatrix}
R_1&R_2\\
R_3&R_4
\end{pmatrix}\quad\mbox{ takes the form }\quad\begin{pmatrix}
R_1&R_2\\
0&R_4
\end{pmatrix},
$$
i.e., $R_3=0$ or, equivalently, $\pi_2\bigl(R(\lambda,\calA)\vect x0\bigr)=0$ for all $x\in E$. Take $x\in E$ and consider the vector
$$
R(\lambda,\calA)\vect x0-\bigvect{R(\lambda,A)x}{0}=\bigvect{R_1x}{R_3x}-\bigvect{R(\lambda,A)x}{0},
$$
which belongs to $\ker(\lambda-\calA)=\{0\}$. Indeed, we have
\[
\pi_1\left((\lambda-\calA)\bigvect{R_1x-R(\lambda,A)x}{R_3x}\right)=\pi_1\vect x0-x=0
\]
and by (i)
\[
\pi_2\left((\lambda-\calA)\bigvect{R_1x-R(\lambda,A)x}{R_3x}\right)=\pi_2\vect x0-\pi_2\vect x0=0.
\]
Hence $R(\lambda, A)x=R_1x$ and $R_3x=0$, and the proof is completed.
\end{proof}

\subsection{Stability Conditions for Matrix Products}

We recall that the underlying idea of our approach is to split a given operator matrix $\calA$ generating a semigroup $\cTtt$ on a product space $\calE$ into a sum $\calA=\calA_1+\calA_2$ of simpler, i.e. triangular, matrices $\calA_i$, $i=1,2$, and then compute $\calT(t)$ using some (e.g. the Trotter) product formula. Here the crucial hypothesis for convergence is a stability condition on the products of the triangular semigroups $\cTitt$, see \eqref{eq:cond_stab_Lie} and \eqref{eq:cond_stab_weighted} in  Theorem~\ref{prop-splitting}.

In this section we will consider three types of such splittings and deduce conditions ensuring that the related stability conditions are satisfied.
We start by considering two operator matrix semigroups of upper triangular form and ask for conditions ensuring that the associated stability condition for the product is satisfied. Let us investigate first the stability condition \eqref{eq:cond_stab_Lie} for the sequential splitting \eqref{eq:conv_Lie} and the Strang splitting \eqref{eq:conv_Strang}.
We remark here that the Strang splitting is  precisely then stable, when  the sequential splitting is. Furthermore, the stability assumption as in  \eqref{eq:cond_stab_Lie} is equivalent to
\begin{equation*}
\bigl\|\left(\calT_2(\tfrac tn)\calT_1(\tfrac tn)\right)^k\bigr\| \leq M \ee^{\omega \frac tn k} \quad\text{for all }t\ge0\text{ and } n,k\in\NN.
\end{equation*}
This is trivially true for all splittings considered in this paper and will be used without further reference (replace $t$ by $\frac {nt}k$ and interchange the roles of $n$ and $k$). The equivalence of the estimates above is even true for more general finite difference schemes, the special splitting structure plays no role here.
\begin{theorem}\label{stability-triangular}
Suppose that for $i=1,2$ the matrix $\calA_i$ generates on $\calE=E\times F$ the semigroup $(\calT_i(t))_{t\geq 0}$ of upper triangular form
$$
\calT_i(t)=
\begin{pmatrix}
T_i(t)& R_i(t)\\
0& S_i(t)
\end{pmatrix}.
$$
If there exist $M'\ge1$, $K>0$ and $\omega'\in\RR$ such that for all $i=1,2$, $t\ge0$ and $n\in\NN$
\begin{enumerate}[(i)]
\item $\|(T_2(\frac tn)T_1(\frac tn))^n\|\leq M'\ee^{\omega' t}$ and  $\|(S_2(\frac tn)S_1(\frac tn))^n\|\leq M'\ee^{\omega't }$,
\item $\|R_i(t)\|\leq K t\cdot \ee^{\omega' t}$,
\end{enumerate}
then there are $M\geq 1$ and $\omega\in\RR$ such that
\begin{equation}\label{eq:stab1}
\bigl\|\bigl(\calT_2(\tfrac tn)\calT_1(\tfrac tn)\bigr)^n\bigr\|\leq M\ee^{\omega t}
\quad\text{for all $t\ge0$ and $n\in\NN$.}
\end{equation}
\end{theorem}

\begin{proof} Since strongly continuous semigroups are exponentially bounded, we can choose $M'\ge1$ and $\omega'\in\RR$ without loss of generality so that
\begin{equation*}
\|T_i(t)\|\leq M'\ee^{\omega' t}, \text{ and } \quad \|S_i(t)\|\leq M'\ee^{\omega' t}
\end{equation*}
are satisfied for $i=1,2$.

\noindent For $h\geq 0$ calculate the product
\begin{align*}
\calT_2(h)\calT_1(h)&=\begin{pmatrix}
T_2(h)&R_2(h)\\0&S_2(h)
\end{pmatrix}\begin{pmatrix}
T_1(h)&R_1(h)\\0&S_1(h)
\end{pmatrix}\\
&=\begin{pmatrix} T_2(h)T_1(h)&T_2(h)R_1(h)+R_2(h)S_1(h)\\ 0 & S_2(h)S_1(h)\end{pmatrix},
\end{align*}
and set $R(h):=T_2(h)R_1(h)+R_2(h)S_1(h)$. This implies
\begin{align*}
\bigl(\calT_2(h)\calT_1(h)\bigr)^2=\begin{pmatrix} \bigl(T_2(h)T_1(h)\bigr)^2&T_2(h)T_1(h)R(h)+R(h)S_2(h)S_1(h)\\ 0 & \bigl(S_2(h)S_1(h)\bigr)^2\end{pmatrix},
\end{align*}
and  by induction one can show that
\begin{equation}\label{eq:power-k}
\bigl(\calT_2(h)\calT_1(h)\bigr)^k =
\begin{pmatrix} \bigl(T_2(h)T_1(h)\bigr)^k & (\star)\\
0 & \bigl(S_2(h)S_1(h)\bigr)^k
\end{pmatrix},
\end{equation}
where
\begin{align}\label{eq:power-k2}
(\star)=\sum_{j=0}^{k-1} \bigl(T_2(h)T_1(h)\bigr)^j R(h)\bigl(S_2(h)S_1(h)\bigr)^{k-1-j}.
\end{align}
In order to prove \eqref{eq:stab1}, we only have to show the exponential estimate for $(\star)$,
the other entries of the product  fulfill such estimates by assumption. Since $\|R(h)\|\le2M'Kh\,\ee^{2\omega' h}$, this implies
\begin{align*}
\|(\star)\|
&\leq \sum_{j=0}^{k-1} \left\|\bigl(T_2(h)T_1(h)\bigr)^j\right\|\cdot\|R(h)\|
  \cdot\left\|\bigl(S_2(h)S_1(h)\bigr)^{k-1-j}\right\|\\
&\leq 2M'^3Kh\sum_{j=0}^{k-1}\ee^{\omega' j h}\ee^{2\omega' h} \ee^{\omega' (k-1-j)h}
=2M'^3Khk\,\ee^{\omega'(k+1)h}.
\end{align*}
If we set $h=\frac tn$ and $k=n$ we get for $M:={2M'}^3K$ and $\omega:=\omega'+|\omega'|+1$
$$
\|(\star)\|\leq Mt\,\ee^{(\omega'+ |\omega'|)t}\leq M\ee^{\omega t}.
$$
This completes the proof.
\end{proof}

\noindent In the same spirit and using analogous calculations, we can investigate the stability condition \eqref{eq:cond_stab_weighted} for the weighted splitting \eqref{eq:conv_weighted}.

\begin{theorem}\label{prop:stability-triangular_weighted}
Suppose that for $i=1,2$ the matrix $\calA_i$ generates on $\calE=E\times F$ the semigroup $(\calT_i(t))_{t\geq 0}$ of upper triangular form
$$
\calT_i(t)=
\begin{pmatrix}
T_i(t)& R_i(t)\\
0& S_i(t)
\end{pmatrix}.
$$
If there exist $M'\ge1$, $K>0$ and $\omega'\in\RR$ such that for all $i=1,2$, $t\ge0$ and $n\in\NN$
\begin{enumerate}[(i)]
\item $\|\frac{1}{2^n}(T_1(\frac tn)T_2(\frac tn)+T_2(\frac tn)T_1(\frac tn))^n\|\leq M'\ee^{\omega' t}$ and \\ $\|\frac{1}{2^n}(S_1(\frac tn)S_2(\frac tn) + S_2(\frac tn)S_1(\frac tn))^n\|\leq M'\ee^{\omega' t}$,
\item $\|R_i(t)\|\leq K t\cdot \ee^{\omega' t}$,
\end{enumerate}
then there are $M\geq 1$ and $\omega\in\RR$ such that
\begin{equation}\label{eq:stab2}
\left\|\frac{1}{2^n}\bigl(\calT_1(\tfrac tn)\calT_2(\tfrac tn) + \calT_2(\tfrac tn)\calT_1(\tfrac tn)\bigr)^n\right\|\leq M\ee^{\omega t}
\end{equation}
for all $t\ge0$ and $n\in\NN$.
\end{theorem}

\begin{proof} Again, since strongly continuous semigroups are exponentially bound\-ed, we can choose $M'\ge1$ and $\omega'\in\RR$ without loss of generality so that
\begin{equation*}
\|T_i(t)\|\leq M'\ee^{\omega' t}, \text{ and } \quad \|S_i(t)\|\leq M'\ee^{\omega' t}
\end{equation*}
are satisfied for $i=1,2$.

\noindent Using the computations of the proof of Theorem \ref{stability-triangular} we obtain for $h\geq 0$  that
\begin{multline*}
\calT_1(h)\calT_2(h)+ \calT_2(h)\calT_1(h)\\
=\begin{pmatrix} T_1(h)T_2(h)+T_2(h)T_1(h)&(\star\star)\\ 0 & S_1(h)S_2(h)+S_2(h)S_1(h)\end{pmatrix},
\end{multline*}
where $(\star\star)=T_1(h)R_2(h)+R_1(h)S_2(h)+T_2(h)R_1(h)+R_2(h)S_1(h)$.

Let $R'(h):=T_1(h)R_2(h)+R_1(h)S_2(h)+T_2(h)R_1(h)+R_2(h)S_1(h)$. Then by induction one can verify the identity
\begin{multline}\label{eq:power-k-w}
\bigl(\calT_1(h)\calT_2(h)+\calT_2(h)\calT_1(h)\bigr)^k  \\
= \begin{pmatrix} \bigl(T_1(h)T_2(h)+T_2(h)T_1(h)\bigr)^k & (\star\star)_k\\
0 & \bigl(S_1(h)S_2(h)+S_2(h)S_1(h)\bigr)^k
\end{pmatrix},
\end{multline}
where
\begin{align}\label{eq:power-k2-w}
(\star\star)_k=\sum_{j=0}^{k-1} \bigl(T_1(h)T_2(h)&+T_2(h)T_1(h)\bigr)^j R'(h)\cdot \notag \\ &\cdot\bigl(S_1(h)S_2(h)+S_2(h)S_1(h)\bigr)^{k-1-j}.
\end{align}
In order to prove \eqref{eq:stab2}, we only have to show the exponential estimate for $(\star\star)_k$,
the other entries of the product  fulfill such estimates by assumption. Since $\|R'(h)\|\le 4M'Kh\,\ee^{2\omega' h}$, this implies
\begin{align*}
\|(\star\star)_k\|
&\leq \sum_{j=0}^{k-1} \left\|\bigl(T_1(h)T_2(h)+T_2(h)T_1(h)\bigr)^j\right\|\cdot\|R'(h)\|\\
&\qquad \cdot\left\|\bigl(S_1(h)S_2(h)+S_2(h)S_1(h)\bigr)^{k-1-j}\right\| \\
&\leq 4M'^3Kh\sum_{j=0}^{k-1}\ee^{\omega' j h} 2^j \ee^{2\omega' h} \ee^{\omega' (k-1-j)h}2^{k-1-j}\\
&=2M'^3Khk\,\ee^{\omega'(k+1)h}2^k.
\end{align*}
If we set $h=\frac tn$ and $k=n$ we get for $M:={2M'}^3K$ and $\omega:=\omega'+|\omega'|+1$
$$
\|(\star\star)_n\|\leq Mt\,\ee^{(\omega' + |\omega'|) t} 2^n \leq M\ee^{\omega t} 2^n.
$$
Combining these estimates, the desired statement \eqref{eq:stab2} follows.
\end{proof}
\noindent Summing up, Theorems \ref{stability-triangular} and \ref{prop:stability-triangular_weighted} show that the stability condition in $(i)$ for the diagonal entries combined with the growth estimate in $(ii)$ imply stability for the matrix products. In the next subsection we will come back to condition $(ii)$.

But first we consider the following stability result for the Trotter, Strang and weighted splitting, which does not make use of a special matrix structure. However, in Subsection~\ref{sec:ABVS} we will apply them in the context of matrix decompositions.

\begin{proposition}\label{stab-A-bdd}
Let $\calA$ generate a semigroup $\cTtt$ on the Banach space $\calE$ and denote by $\cStt$ the semigroup generated by $\calC\in\calL(\calE)$, i.e., $\calS(t)=\ee^{t\calC}$.
Then there exist constants $M\ge1$ and $\omega\in\RR$ such that for all $t\ge0$ and $n\in\NN$
\begin{align*}
\bigl\|\bigl(\calS(\tfrac tn)\calT(\tfrac tn)\bigr)^n\bigr\|&\leq M\ee^{\omega t}
\quad\text{and}\\
\left\|\frac{1}{2^n}\bigl(\calS(\tfrac tn)\calT(\tfrac tn) + \calT(\tfrac tn)\calS(\tfrac tn)\bigr)^n\right\| &\leq M \ee^{\omega t}.
\end{align*}

\end{proposition}

\begin{proof} By Engel and Nagel \cite[Lemma II.3.10]{EN:00}, there exists an equivalent norm $\N\cdot\N$ on $\calE$ such that $\cTtt$ is quasi-dissipative for $\N\cdot\N$, i.e., satisfies an estimate
\[
\N\cTt\N\le\ee^{\omega't}\quad\text{for all } t\ge0
\]
and some $\omega'\in\RR$. Moreover,
\[
\N\cSt\N\le\ee^{\N\calC\N t}\quad\text{for all } t\ge0,
\]
where $\N\calC\N$ denotes the operator norm of $\calC\in\calL(\calE)$ induced by $\N\cdot\N$. Since $\|\cdot\|\simeq\N\cdot\N$ there exist $m',M'>0$ such that $m'\|\calx\|\le\N\calx\N\le M'\|\calx\|$ for all $\calx\in\calE$
and hence
\begin{align*}
\left\|\bigl(\calS(\tfrac tn)\calT(\tfrac tn)\bigr)^n\calx\right\|
  &\le\tfrac1{m'}\cdot\blN\bigl(\calS(\tfrac tn)\calT(\tfrac tn)\bigr)^n\calx\brN\\
  &\le\tfrac{1}{m'}\cdot\ee^{\N\calC\N t}\cdot\ee^{\omega't}\cdot\N\calx\N\\
  &\leq \tfrac{M'}{m'}\cdot\ee^{(\N\calC\N+\omega')t}\cdot\|\calx\|
\end{align*}
for all $t\ge0$ and $\calx\in\calE$. This implies the first estimate for $M:=\tfrac{M'}{m'}$ and $\omega:=\omega'+\N\calC\N$. The second estimate follows similarly from
\begin{align*}
\left\|\bigl(\calT(\tfrac tn)\calS(\tfrac tn)+\calS(\tfrac tn)\calT(\tfrac tn)\bigr)^n\calx\right\|
  &\le\tfrac{1}{m'}\cdot\left(2\cdot
  \ee^{\N\calC\N\frac tn}\cdot\ee^{\omega'\frac tn}\right)^n\cdot\N\calx\N\\
  &\leq 2^n\cdot\tfrac{M'}{m'}\cdot\ee^{(\N\calC\N+\omega')t}\cdot\|\calx\|
\end{align*}
for the same constants $M$ and $\omega$ as above.
\end{proof}

\noindent The previous result applies in particular to the splitting
\[
\calA=\calA_0+\calC
\quad\text{where}\quad
\calC=
\begin{pmatrix}
  0&0\\C&0
\end{pmatrix}
\]
for some $C\in\calL(E,F)$, if we assume that $\calA_0$ generates a matrix semigroup on $\calE=E\times F$. In this case the semigroup $\cStt$ generated by $\calC$ is given by
\begin{equation}\label{exp(tC)}
\calS(t)=
\begin{pmatrix}
  \Id&0\\tC&\Id
\end{pmatrix}.
\end{equation}

\subsection{Estimates for Triangular Matrix Semigroups}

As we saw in the previous subsection, cf.~condition~(iii) in Theorems~\ref{stability-triangular} and \ref{prop:stability-triangular_weighted}, in order to obtain the desired stability estimates \eqref{eq:stab1} and \eqref{eq:stab2} we need estimates of the type $\|R_i(t)\|\leq Kt$, $i=1,2$, for the upper right entries $R_i(t)$ of $\calT_i(t)$. In this section we will use an approach based on the concept of Favard classes to achieve this goal.

\medskip Recall from  the proof of Proposition~\ref{char-triangular} that given a matrix semigroup $\cTtt$ of a triangular form \eqref{triangular-sgr}
the diagonal entries $\Ttt$ and $\Stt$ are semigroups on $E$ and $F$, respectively.  If $A$ and $B$ denote their generators, we define the diagonal matrix
\[
\calD=
\begin{pmatrix}
A&0\\
0&B
\end{pmatrix},\quad
\dom(\calD)=\dom(A)\times\dom(B)
\]
which generates the diagonal semigroup $\cStt$ given by
\[
\calS(t)=
\begin{pmatrix}
T(t)&0\\
0& S(t)
\end{pmatrix}.
\]
Moreover, we denote by $(\calS_{-1}(t))_{t\ge0}$ the extrapolated semigroup
\[
\calS_{-1}(t)=
\begin{pmatrix}
T_{-1}(t)&0\\
0& S_{-1}(t)
\end{pmatrix}
\]
with generator
\[
\calD_{-1}=
\begin{pmatrix}
A_{-1}&0\\
0&B_{-1}
\end{pmatrix},\quad
\dom(\calD_{-1})=\dom(A_{-1})\times\dom(B_{-1}),
\]
see Engel and Nagel \cite[Section~II.5]{EN:00}.
For the convenience of the reader we collect here some
facts concerning Favard classes of semigroup generators. A much more detailed account can be found in Engel and Nagel \cite[Section II.5.b]{EN:00}.

\begin{definition}\label{def-fc} Let $\Ttt$ be a strongly continuous semigroup on a
Banach space $E$ with generator $A$. Then we define its \textit{Favard class} (or \textit{space}) as
\[\FeA:=\Bigl\{x\in E:\sup_{t\in(0,1]}t^{-1}\cdot\bigl\|T(t)x-x\bigr\|<\infty\Bigr\}\subset E,\]
which becomes a Banach space with respect to the norm
\[\|x\|_\FeA:=\|x\|+\sup_{t\in(0,1]}t^{-1}\cdot\bigl\|T(t)x-x\bigr\|.\]
\end{definition}

\noindent We note that for reflexive Banach spaces $E$ one always has $\FeA=D(A)$ (see Engel and Nagel \cite[Corollary II.5.21]{EN:00}), hence Favard spaces are interesting only in nonreflexive spaces.

One can define the Favard space $\FnA=\Fav_1(A_{-1})$ for the extrapolated semigroup $\Tmett$ with generator $A_{-1}$ in a similar manner. Using these notations we have the following result.

\begin{proposition}\label{comparison-favard} Let $\cTtt$ be a triangular semigroup of the form \eqref{triangular-sgr} on the product space $\calE=E\times F$ with generator $\calA$.
Then the following assertions are equivalent.
\begin{enumerate}[(a)]
\item There exists $K>0$ such that $\|R(t)\|\le Kt$ for all $t\in[0,1]$.
\item[(a')] There exists $K>0$, $\omega\in\RR$ such that $\|R(t)\|\le Kt\cdot \ee^{\omega t}$ for all $t\ge0$.
\item There exists  $P\in\calL(F,\FnA)$ such that
$\calA=(\calD_{-1}+\calP)|_\calE$ where
\[
\calP=
\begin{pmatrix}
0&P\\0&0
\end{pmatrix}.
\]
\item For some/all $\lambda\in\rho(\calD)$ there exists $D_\lambda\in\calL(F,\FeA)$ such that
\[
\lambda-\calA=
(\lambda-\calD)
\begin{pmatrix}
\Id&-D_\lambda\\0&\Id
\end{pmatrix}.
\]
\end{enumerate}
\end{proposition}

\begin{proof}
The equivalence of (a) and (b) follows from Engel and Nagel \cite[Theorem III.3.9]{EN:00}, while (b) and (c) are equivalent by \cite[Proposition III.3.18.(ii)]{EN:00}. Finally, (a) and (a') are equivalent since every strongly continuous semigroup is exponentially bounded.
\end{proof}

\section{Applications}\label{sec:applications}

\noindent
In this section we will show how our abstract results apply to inhomogeneous Cauchy problems as well as to systems with boundary feedback.

\subsection{Inhomogeneous Abstract Cauchy Problems}

Consider the inhomogeneous Cauchy problem
\[\tag{iACP}\label{eq:acpspl_inhom}
\begin{cases}
\frac{\dd }{\dd t}\, u(t)=Au(t)+f(t)\quad\text{for } t\ge0,\\
u(0)=u_0.
\end{cases}
\]
for a linear operator $(A,D(A))$ on a Banach space $E$. For operator splitting methods applied to this problem, see Bj\o rhus \cite{Bjorhus} and Ostermann and Schratz \cite{Ostermann-Schratz}.

A standard method to tackle this problem is to rewrite it as a homogeneous one like \eqref{ACP} in the product space $\calE:=E\times F(\RR_+;E)$ for the operator matrix
\begin{equation*}
\calA=\begin{pmatrix}
A & \delta_0 \\
0 & \frac{\dd}{\dd s}
\end{pmatrix}\quad\mbox{with diagonal domain}\quad\dom(\calA)=\dom(A)\times F_1(\RR_+;E).
\end{equation*}
Here $F(\RR_+;E)$ denotes a space of $E$-valued functions defined on $\RR_+$ on which the left-shift semigroup $(L(t))_{t\geq 0}$ is strongly continuous. Moreover, $\frac{\dd}{\dd s}$ with domain $F_1(\RR_+;E)$ denotes the generator of $(L(t))_{t\geq 0}$, and $\delta_0(f):=f(0)$ is the point evaluation at $0$. The main choices for $F:=F(\RR_+;E)$ are $F=\C_0(\RR_+;E)$ which implies $F_1(\RR_+;E)=C_0^1(\RR_+;E)$ or $F=\Ell^1(\RR_+;E)$ for which $F_1(\RR_+;E)=\W^{1,1}(\RR_+;E)$ follows.
Then the inhomogeneous equation \eqref{eq:acpspl_inhom} is equivalent to the abstract Cauchy problem
\begin{equation*}
\begin{cases}
\tfrac {\dd }{\dd t}\calu(t)=\calA\,\calu(t)\quad\text{for }t\ge0,\\
\calu(0)=\vect{u_0}{f}.
\end{cases}
\end{equation*}
For the details we refer to Engel and Nagel \cite[Section VI.7]{EN:00}. Here we only mention that in both cases $\calA$ generates a strongly continuous semigroup on $\calE$. In the $\C_0$-case this easily follows by bounded perturbation (see below) while in the ${\Ell^1}$-case this is shown in \cite[Proposition VI.7.5]{EN:00}.

\subsubsection{Stability in ${F=\C_0(\RR_+;E)}$}
To see that the operator matrix $\calA$ is actually a generator in case $F=\C_0(\RR_+;E)$, note that
\begin{equation*}
\calA_0=\begin{pmatrix}
A & 0 \\
0 & \frac{\dd}{\dd s}
\end{pmatrix}\quad\mbox{with diagonal domain}\quad\dom(\calA_0)=\dom(A)\times D(\tfrac{\dd}{\dd s})
\end{equation*}
for $D(\frac{\dd}{\dd s})=\C_0^1(\RR_+;E)$ generates the semigroup
\begin{equation*}
\calT_0(t):=\begin{pmatrix}
T(t) & 0 \\
0 & L(t)
\end{pmatrix},
\end{equation*}
where $(T(t))_{t\geq 0}$ is the semigroup generated by $A$.
Since $\delta_0:F\to E$ is bounded, $\calA$ is a bounded perturbation of $\calA_0$, hence it is a generator. To get a formula for the semigroup $(\calT(t))_{t\geq 0}$ generated by $\calA$, note that by Proposition~\ref{char-triangular} the semigroup $(\calT(t))_{t\geq0}$ must be upper triangular, say
$$
 \calT(t)=\begin{pmatrix}
 T_1(t)&T_2(t)\\
 0&T_3(t)
 \end{pmatrix}.
$$
By the variation of constants formula (see e.g., Engel and Nagel \cite[Section III.1]{EN:00}) we obtain
\begin{align*}
\calT(t)\vect{x}{f}&=\calT_0(t)\vect{x}{f}+\int_0^t \calT(t-s)\begin{smatrix}0&\delta_0\\0&0\end{smatrix}\calT_0(s)\vect{x}{f}\dd s\\
&=\vect{T(t)x}{L(t)f}+\int_0^t \vect{T_1(t-s)f(s)}{0}\dd s.
\end{align*}
If we take $f=0$, we get $T_1(t)=T(t)$, and hence for all $f\in F$ we have
$$
T_2(t)f=\int_0^t T(t-s)f(s)\dd s.
$$
Moreover, $T_3(t)=L(t)$.

Now we want to apply the sequential splitting to the problem
\begin{equation}\label{eq:inhom_beforesplit}
\begin{cases}
\frac{\dd }{\dd t} u(t)=(A_1+A_2)u(t)+(f_1+f_2)(t)\quad\text{for }t\ge0,\\
u(0)=u_0,
\end{cases}
\end{equation}
where we have written the inhomogeneity already in a form corresponding to the splitting procedure. Namely, choosing a time step $h=\frac{t}{n}$, we first solve the equation
\begin{equation*}
\begin{cases}
\frac{\dd }{\dd t} v(h)=A_1v(h)+f_1(h),\\
v(0)=u_0,
\end{cases}
\end{equation*}
then using the result we solve the equation
\begin{equation*}
\begin{cases}
\frac{\dd }{\dd t} w(h)=A_2 w(h)+f_2(h),\\
w(0)=v(h).
\end{cases}
\end{equation*}
Setting $u_h=w(h)$, we repeat this procedure $n$ times and call $u_{nh}$ the (sequential) split solution corresponding to the equation \eqref{eq:inhom_beforesplit}.

Clearly, by the preparations in the beginning of this section, we can reformulate \eqref{eq:inhom_beforesplit} as a homogeneous abstract Cauchy problem
\begin{equation*}
\begin{cases}
\frac{\dd}{\dd t}\calu(t)=(\calA_1+\calA_2)\,\calu(t)\quad\text{for }t\ge0,\\
\calu(0)=\left(\begin{smallmatrix}{u_0}\\{f_1}\\{f_2}\end{smallmatrix}\right)
\end{cases}
\end{equation*}
on the product space
\[
\calE=E\times F\times F,
\]
for $F=\C_0(\RR_+;E)$ and the operators
\begin{equation}\label{eq:decomp-A-C0}
\begin{aligned}
\calA_1&:=\begin{pmatrix}
A_1 & \delta_0 & 0 \\
0 & \frac{\dd}{\dd s} & 0 \\
0 & 0 & 0
\end{pmatrix}, \quad \dom(\calA_1)=\dom(A_1)\times D(\tfrac{\dd}{\dd s}) \times F,\\
\calA_2&:=
\begin{pmatrix}A_2 & 0 & \delta_0 \\
0 & 0 & 0 \\
0 & 0 & \frac{\dd}{\dd s}
\end{pmatrix}, \quad \dom(\calA_2)=\dom(A_2)\times F\times D(\tfrac{\dd}{\dd s})
\end{aligned}
\end{equation}
for $D(\frac{\dd}{\dd s})=\C_0^1(\RR_+;E)$.
By the above, the semigroups generated by these operators take the form
\begin{equation}\label{eq:Tt+St}
\mathcal{T}_1(t)=
\left(
\begin{array}{ccc}
T_1(t) & Q_1(t) & 0 \\
0 & L(t) & 0 \\
0 & 0 & I
\end{array}
\right), \qquad
\mathcal{T}_2(t)=
\left(
\begin{array}{ccc}
T_2(t) & 0 & Q_2(t) \\
0 & I & 0 \\
0 & 0 & L(t)
\end{array}
\right),
\end{equation}
where $(T_1(t))_{t\geq 0}$ and $(T_2(t))_{t\geq0}$ denote the semigroups generated by $A_1$ and $A_2$, respectively, $(L(t))_{t\geq0}$ is the left-shift on $\C_0(\RR_+;E)$, and
$$
Q_i(t)f=\int_0^t T_i(t-s)f_i(s)\dd s\quad \text{for } i=1,2\text{ and }f\in F.
$$
Note that with this notation, the sequential splitting is given by the Trotter product formula
$$
u_{nh} = \pi_1 \left(\mathcal{T}_2(h)\mathcal{T}_1(h)\right)^n \left(\begin{smallmatrix}{u_0}\\{f_1}\\{f_2}\end{smallmatrix}\right).
$$
The next result establishes the stability condition \eqref{eq:cond_stab_Lie}, and hence the convergence for the Trotter and Strang product formulas with respect to the splitting $\calA=\calA_1+\calA_2$ for $\calA_1$ and $\calA_2$ defined by \eqref{eq:decomp-A-C0}.

\begin{proposition}\label{t:stab-C0}
Suppose that for some $M'\ge1$ and $\omega'\geq 0$ one has
$$
\|(T_2(\tfrac tn)T_1(\tfrac tn))^n\|\leq M'\ee^{t\omega'}$$
for all $t\ge 0$ and $n\in \NN$. Then there exist $M\ge1$ and $\omega\in\RR$ such that
\begin{equation}\label{eq:stab-C0}
\bigl\|\bigl(\calT_2(\tfrac tn)\calT_1(\tfrac tn)\bigr)^n\bigr\|\leq M\ee^{t\omega}\quad\mbox{holds for all $t\geq 0$ and $n\in\NN$}.
\end{equation}
Moreover, the product formulas \eqref{eq:conv_Lie} and \eqref{eq:conv_Strang} described in Theorem~\ref{prop-splitting} with respect to the operator splitting $\calA=\calA_1+\calA_2$ for $\calA_{1}$, $\calA_2$ defined by \eqref{eq:decomp-A-C0} hold.
\end{proposition}

\begin{proof}
Since $(\calT_1(t))_{t\geq0}$ and $(\calT_2(t))_{t\geq0}$ are bounded perturbations of diagonal semigroups, the claim follows from Theorem~\ref{stability-triangular} and Proposition~\ref{comparison-favard}.
\end{proof}

\noindent In a similar way we obtain the following results concerning the weighted splitting.

\begin{proposition}\label{t:stab-C0-weighted}
Suppose that for some $M'\ge1$ and $\omega'\geq 0$ one has
$$\bigl\|\tfrac{1}{2^n}(T_1(\tfrac tn)T_2(\tfrac tn)+T_2(\tfrac tn)T_1(\tfrac tn))^n\bigr\|\leq M'\ee^{t\omega'}$$
for all $t\ge 0$ and $n\in \NN$. Then there exist $M\ge1$ and $\omega\in\RR$ such that
\begin{equation}\label{eq:stab-C0-weighted}
\Bigl\|\frac{1}{2^n}\bigl(\calT_1(\tfrac tn)\calT_2(\tfrac tn)+\calT_2(\tfrac tn)\calT_1(\tfrac tn)\bigr)^n\Bigr\|\leq M\ee^{t\omega}
\end{equation}
holds for all $t\geq 0$ and $n\in\NN$. Moreover, the product formula \eqref{eq:conv_weighted} described in Theorem~\ref{prop-splitting} with respect to the operator splitting $\calA=\calA_1+\calA_2$ for $\calA_{1}$, $\calA_2$ defined by \eqref{eq:decomp-A-C0} holds.
\end{proposition}

\begin{proof}
The proof follows similarly as the one of Proposition~\ref{t:stab-C0} from Theorem~\ref{prop:stability-triangular_weighted} and Proposition~\ref{comparison-favard}.
\end{proof}
\noindent Note that the condition $\omega'\geq 0$ is neither a restriction, nor crucial, and was chosen only to simplify our calculations in the following subsection.

\subsubsection{Stability in $F=\Ell^p(\RR_+;E)$}
Our aim is now to prove that Propositions~\ref{t:stab-C0} and \ref{t:stab-C0-weighted} remain true if we replace the space $F=\C_0(\RR_+;E)$ by $F=\Ell^p(\RR_+;E)$ for some $1\leq p<\infty$.
This is not straightforward since in the $\C_0$-case the stability condition \eqref{eq:stab-C0} follows by bounded perturbation. However, in the $\Ell^{p}$-case the perturbation
\[
\delta_0:D\bigl(\tfrac{\dd}{\dd s}\bigr)=\W^{1,p}(\RR_+;E)\subset F
\to E
\]
is unbounded on $F$ and hence it is not guaranteed in general that the off-diagonal perturbing term $R(t)$ is $\calO(t)$ as $t\to0^+$. Nevertheless, due to a particular additivity property of the norm in $\Ell^{1}$, stability prevails also in this case. To show this, suppose that the conditions of Proposition \ref{t:stab-C0-weighted} are satisfied.
First we group the entries of $\calT_i(t)$, $i=1,2$, from \eqref{eq:Tt+St} and obtain the $2\times 2$-block matrices
\begin{align*}
\calT_1(t)&=\left(
\begin{MAT}(b){c1cc}
T_1(t) &Q_1(t) & 0\\1
  0   & L(t)& 0\\
  0  &0 &I\\
\end{MAT}
\right)
=:\begin{pmatrix}
T_1(t)& R_1(t)\\
0& S_1(t)
\end{pmatrix},
\\
\calT_2(t)&=\left(
\begin{MAT}(b){c1cc}
T_2(t) &0 & Q_2(t)\\1
  0   & I& 0\\
  0  &0 &L(t)\\
\end{MAT}
\right)
=:\begin{pmatrix}
T_2(t)& R_2(t)\\
0& S_2(t)
\end{pmatrix}.
\end{align*}
Here, as in the previous case, $(L(t))_{t\ge0}$ denotes the left-shift semigroup  which is now defined on the space $F=\Ell^p(\RR_+;E)$ and has generator $\frac{\dd}{\dd s}$ with domain $D(\frac{\dd}{\dd s})=\W^{1,p}(\RR_+;E)$. Then from
\eqref{eq:power-k} and \eqref{eq:power-k2} we obtain that
\begin{equation*}
\bigl(\mathcal{T}_2(h)\mathcal{T}_1(h)\bigr)^k =
\begin{pmatrix} \bigl(T_2(h)T_1(h)\bigr)^k & (*) & (**) \\
0 & L(kh) & 0 \\
0 & 0 & L(kh) \end{pmatrix},
\end{equation*}
where
\begin{align*}
(*) &= \sum_{j=0}^{k-1} \bigl(T_2(h)T_1(h)\bigr)^j T_2(h)Q_1(h)L\bigl((k-j-1)h\bigr),
\intertext{and}
(**) &= \sum_{j=0}^{k-1} \bigl(T_2(h)T_1(h)\bigr)^j Q_2(h)L\bigl((k-j-1)h\bigr).
\end{align*}
In the $\C_0$-case $Q_1(h)$ and $Q_2(h)$ were $\calO(h)$ as $h\to0^+$ and hence rather crude estimates for the sums $(*)$ and $(**)$ already implied stability.
In the present situation we have to be more careful and estimate
\begin{align*}
\|(**)\| &\leq \sum_{j=0}^{k-1} \left\|\bigl(T_2(h)T_1(h)\bigr)^j\right\| \cdot \left\|Q_2(h)L\bigl((k-j-1)h\bigr)\right\| \\
&\leq M' \ee^{\omega' k h}  \sum_{j=0}^{k-1} \left\|Q_2(h)L\bigl((k-j-1)h\bigr)\right\|.
\end{align*}
Note now that since $f\in\Ell^p$, it is also in $\Ell^1_{\text{loc}}$. Using the additivity of the $\Ell^1$-norm with respect to the domain of integration we obtain for $f\in F$
\begin{align*}
\sum_{j=0}^{k-1} &\left\|Q_2(h)L\bigl((k-j-1)h\bigr)f\right\|\\
&=\sum_{j=0}^{k-1}\Bigl\|\int_0^h T_2(h-s)\left[L\bigl((k-j-1)h\bigr)f\right](s)\dd s\Bigr\|\\
&=\sum_{j=0}^{k-1}\Bigl\|\int_0^h T_2(h-s)f(s+(k-j-1)h))\dd s\Bigr\|\\
&\leq M'\ee^{\omega' h}\sum_{j=0}^{k-1}\int_{jh}^{(j+1)h} \|f(s)\|\dd s = M'\ee^{\omega' h}\int_0^{kh}\|f(s)\| \dd s \\
& \leq M'\ee^{\omega' h} (kh)^{1-\frac{1}{p}}\left(\int_0^{kh}\|f(s)\|^p \dd s\right)^{\frac{1}{p}}
 \leq  M''\ee^{\omega'' kh}\|f\|.
\end{align*}%
This implies that there exists $M\geq 1$ and $\omega\geq 0$ so that
\begin{equation*}
\|(**)\| \leq M \ee^{\omega kh}\quad\mbox{ and similarly }\quad  \|(*)\| \leq M \ee^{\omega kh}
\end{equation*}
hold.

\begin{remark}
In the same spirit, the stability of the weighted splitting can also be established. Since the proof is straightforward and would be only a repetition of what we had done so far, we omit it.
\end{remark}

\noindent We therefore obtain the following results.

\begin{proposition} Propositions~\ref{t:stab-C0} and \ref{t:stab-C0-weighted} prevail for $F=\Ell^p(\RR_+;E)$ and $\frac{\dd}{\dd s}$ with domain $D(\frac{\dd}{\dd s})=\W^{1,p}(\RR_+;E)$.
\end{proposition}

Summing up both cases, we have established that the splitting for the inhomogeneous abstract Cauchy problem with a $\C_0(\RR_+;E)$ or a $\Ell^p(\RR_+;E)$ inhomogeneity is stable if the splitting for the  associated homogeneous problem is stable.

\subsection{Abstract Boundary Feedback Systems}\label{sec:ABVS}
Let $E$ and $\partial E$ be Banach spaces and let the operators $A_m:\dom(A_m)\subseteq E\to E$, $B:\dom(B)\subseteq\partial E \to \partial E$, $C\in\calL(E,\partial E)$ and $L:\dom(A_m)\to \partial E$ be given. An abstract boundary feedback system is a system of two coupled differential equations of the form
\[\tag{ABFS}\label{ABFS}
\begin{cases}
\tfrac\dd{\dd t}\,u(t)=A_m u(t),& t\ge0,\\
\tfrac\dd{\dd t}\,x(t)=B x(t)+Cu(t),& t\ge0,\\
Lu(t)=x(t),& t\ge0,\\
u(0)=u_0,\; x(0)=x_0,
\end{cases}
\]
where the functions $u$ and $x$ are $E$ and $\partial E$-valued, respectively. We refer to Casarino et al.~\cite{CENN:03} for more details and concrete examples. Now under suitable assumptions (see below) such systems can be rewritten as an abstract Cauchy problem \eqref{ACP}, where the coupling $Lu(t)=x(t)$ of the two equations is coded in the domain of the system operator $\calA$. To proceed we make as in Casarino et al.~\cite[Section~2]{CENN:03} the following
\begin{assumption}\label{assumptions-CENN}
\begin{enumerate}[{(i)}]
\item $A:=A_m|_{\ker L}$ generates a semigroup $(T(t))_{t\geq 0}$ on $E$.
\item $L:\dom (A_m)\to \partial E$ is surjective.
\item $\vect{A_m}{L}:\dom (A_m)\to E\times \partial E$ is closed.
\item $B$ generates a semigroup $(S(t))_{t\geq 0}$ on $\partial E$.
\end{enumerate}
\end{assumption}

\noindent Then by Casarino et al.~\cite[Lemma 2.2]{CENN:03}, the following holds.

\begin{lemma}\label{dirichlet} If $\lambda\in\rho(A)$, then the restriction $L|_{\ker(\lambda-A_m)}:\ker(\lambda-A_m)\to\partial E$
is invertible and its inverse, called \textit{Dirichlet operator},
\[
D_\lambda:=\bigl(L|_{\ker(\lambda-A_m)}\bigr)^{-1}:\partial E\to\ker(\lambda-A_m)\subset E
\]
is bounded.
\end{lemma}

\begin{remark} We note that condition (iii) in Assumption~\ref{assumptions-CENN} can be replaced by
\begin{enumerate}
\item[(iii')] $D_\lambda\in\calL(\partial E,E)$ exists for all $\lambda\in\rho(A)$
\end{enumerate}
which sometimes is easier to verify than the closedness of $\vect{A_m}{L}$.
\end{remark}

\noindent In order to treat \eqref{ABFS} by semigroup methods we
define on $\calE:=E\times\dE$  the operator matrix
\begin{align*}
\calA_C:&=\begin{pmatrix}\Am&0\\C&B
       \end{pmatrix}
\intertext{with domain} D(\calA_C):&=\big\{\tbinom{f}{x}\in
D(\Am)\times D(B):Lf=x\big\}.
\end{align*}

\noindent Then by Casarino et al.~\cite[Section 2]{CENN:03} and by Engel and Nagel \cite{EN:00}, the system \eqref{ABFS} is well-posed if and only if the operator matrix $\calA_C$ generates a strongly continuous semigroup $\cTCtt$ on $\calE$. Moreover, in this case for every initial value $\binom{u_0}{x_0}\in\dom(\calA_C)$ the unique solution of \eqref{ABFS} is given by
\[
\RR_+\ni t\mapsto\pi_1\bigl(\cTCt\tbinom{u_0}{x_0}\bigr)\in E.
\]
In order to apply the splitting approach to this problem we first assume that $C=0$ and decompose $\calA_0=\calA_1+\calA_2$ for
\begin{equation}\label{eq:dec-CENN}
\begin{aligned}
\calA_1:=&\begin{pmatrix}\Am&0\\0&0
       \end{pmatrix}, &&D(\calA_1):=\big\{\tbinom{f}{x}\in
D(\Am)\times \partial E:Lf=x\big\},\\
\calA_2:=&\begin{pmatrix}0&0\\0&B
       \end{pmatrix}, &&D(\calA_2):=E\times D(B).
\end{aligned}
\end{equation}
Then by Casarino at al.~\cite[Corollary~2.9]{CENN:03} the matrix $\calA_1$ is the generator of a strongly continuous semigroup $(\calT_1(t))_{t\geq0}$. Moreover, if $A$ is invertible, then $\calT_1(t)$ is given by
\[
\calT_1(t)=
\begin{pmatrix}
T(t)&\bigl(\Id-T(t)\bigr)D_0\\
0&\Id
\end{pmatrix}.
\]
On the other hand, also $\calA_2$ is a generator of a strongly continuous semigroup $(\calT_2(t))_{t\geq0}$ which can be easily calculated as
\[
\calT_2(t)=
\begin{pmatrix}
\Id&0\\
0& S(t)
\end{pmatrix}.
\]
Now the assumptions (i)--(iii) of Theorem~\ref{stability-triangular} and Theorem~\ref{prop:stability-triangular_weighted} are satisfied if there exists $K>0$ and $\omega\in\RR$ such that for $R_1(t):=(\Id-T(t))D_0$ we have that
\begin{equation}\label{eq:stab-Rt}
\|R_1(t)\|\le Kt\cdot\ee^{\omega t}\quad\text{for all }t\ge0.
\end{equation}
Note that by Proposition~\ref{comparison-favard}, Lemma~\ref{lem-DS} and \cite[Lemma~2.6]{CENN:03}, condition~\eqref{eq:stab-Rt} is equivalent to the assumption
\begin{equation}\label{eq:stab-Rt-fav}
D(\Am)\subset\Fav_1(A).
\end{equation}
This condition can be characterized by the following result of Desch and Schappacher \cite[Theorem 9]{DS:89}.
\begin{lemma}\label{lem-DS} Let the Assumptions~\ref{assumptions-CENN}.(i)--(iii) be satisfied. If $D_\lambda$ denotes the Dirichlet operator introduced in Lemma~\ref{dirichlet}, then the following conditions are equivalent.
\begin{enumerate}[(a)]
\item $D(\Am)\subset\FeA$.
\item $\ker(\lambda-\Am)\subset\FeA$ for some $\lambda\in\rho(A)$.
\item There exist $\gamma>0$ and $\lambda_0\in\RR$ such that
$\|Lx\|\ge \gamma\lambda\cdot\|x\|$ for all $\lambda>\lambda_0$,
$x\in\ker(\lambda-\Am)$.
\item There exist $c>0$ and $w>0$ such that
$\|D_\lambda\|\le c\cdot\lambda^{-1}$ for all $\lambda>w$.
\end{enumerate}
\end{lemma}

\noindent Summing up, we obtain the following.
\begin{corollary}\label{thm:conv-splitt-ABVS}
Let the Assumptions~\ref{assumptions-CENN} be satisfied. If in addition $0\in\rho(A)$, $C=0$ and $D(A_m)\subset\Fav(A)$, then the product formulas \eqref{eq:conv_Lie} and \eqref{eq:conv_Strang} described in Theorem~\ref{prop-splitting} for the decomposition $\calA_0=\calA_1+\calA_2$ and $\calA_{1}$, $\calA_2$ defined by \eqref{eq:dec-CENN} converge to the semigroup $(\calT_0(t))_{t\ge0}$ generated by $\calA_0$.
\end{corollary}

\noindent In the next step we add a non-zero feedback operator $C\in\calL(E,\partial E)$ to our setting. More precisely, we decompose
\begin{equation}\label{eq:dec-CENN-C}
\calA_C= {}\calA_0+\calC\qquad\text{where}\qquad
\calC:={}
\begin{pmatrix}
0&0\\C&0
\end{pmatrix}\in\calL(\calE).
\end{equation}
Then from Proposition~\ref{stab-A-bdd} we obtain the following result.

\begin{corollary}
Let the Assumptions~\ref{assumptions-CENN} be satisfied and let $C\in\calL(\partial E,E)$. Then the product formulas \eqref{eq:conv_Lie}, \eqref{eq:conv_Strang} and \eqref{eq:conv_weighted} for the Trotter, Strang and weighted splitting with respect to the decomposition \eqref{eq:dec-CENN-C} converge to the semigroup $(\calT_C(t))_{t\ge0}$ generated by $\calA_C$.
\end{corollary}

\section*{Acknowledgments}
\noindent
The European Union and the European Social Fund have provided financial support to the project under the grant agreement no. T\'AMOP-4.2.1/B-09/1/KMR-2010-0003. Supported by the OTKA grant Nr. K81403. During the preparation of the paper B.~F.~was supported by the J\'{a}nos Bolyai Research Scholarship of the Hungarian Academy of Sciences. The financial support of the ``Stiftung Aktion \"Osterreich-Ungarn'' is gratefully acknowledged.

\end{document}